\documentclass{amsart}
\usepackage{amsmath}
\usepackage{graphicx}
\usepackage{bm,color}

\title
[Unbounded generalization ...]
{Unbounded generalization of logarithmic representation of infinitesimal generators}

\author{Yoritaka Iwata}

\address[Y. Iwata]{Kansai University, Japan}
\curraddr{Kansai University, Osaka 564-8680, Japan}
\email{iwata$\_$phys@08.alumni.u-tokyo.ac.jp}

\thanks{The author is grateful to Prof. Emeritus Dr. Hiroki Tanabe for fruitful comments.
This work was supported by JSPS KAKENHI Grant No. 17K05440.}
\keywords{Nonlinear equation.}
\subjclass[2000]{46B99}

\theoremstyle{plain}
\newtheorem{theorem}{Theorem}
\newtheorem{corollary}{Corollary}
\newtheorem{lemma}{Lemma}

\begin{document}

\begin{abstract}
The logarithmic representation of infinitesimal generators is generalized to the cases when the evolution operator is unbounded. 
The generalized result is applicable to the representation of infinitesimal generators of unbounded evolution operators, where unboundedness of evolution operator is an essential ingredient of nonlinear analysis.
In conclusion a general framework for the identification between the infinitesimal generators with evolution operators is established.
A mathematical framework for such an identification is indispensable to the rigorous treatment of nonlinear transforms: e.g., transforms appearing in the theory of integrable systems. 
\end{abstract}

\maketitle

\keywords{Keywords: logarithm of operators, resolvent operator, Dunford-Taylor integral \\}

\section{Introduction}
The theory of analytic semigroup~\cite{61kato,60tanabe,61tanabe} provides an integral representation for a certain class of semigroups of operators.
Theory of analytic semigroup actually clarifies the regularity of solutions for evolution equations of parabolic type.
In the theory of analytic semigroup of operators, the exponential functions of unbounded operators are defined by means of the Dunford-Taylor integral~\cite{51taylor}, where the Riesz-Dunford integral~\cite{43dunford} and the Dunford-Taylor integral are the generalized concepts of the Cauchy integral in the complex analysis.
In the preceding work~\cite{17iwata-1}, the logarithmic function of bounded operators are shown to be well-defined, while the theory for defining the logarithmic function of unbounded operators has not been established yet, except for the logarithms of sectorial operators \cite{94boyadzhiev, 01carracedo, 03hasse, 06hasse, 69nollau, 00okazawa-1, 00okazawa-2}.
Note that the evolution operator generated for the free wave equation is a non-sectorial operator (cf. hyperbolic partial differential equations).   

The aim of this paper is to generalize the logarithmic representation of infinitesimal generators by defining the logarithmic functions of unbounded operators without assuming the sectorial property of the evolution operators.
Although the boundedness of evolution operator is assumed in the standard theory of abstract evolution equation, this restriction prevents us to have access to the abstract and general treatments of nonlinear transforms in which the evolution operators (i.e. solutions) are often identified with the infinitesimal generators of another equations.
For the applicable example of nonlinear transforms, the Cole-Hopf transform  \cite{19iwata-2} and the Miura transform  \cite{20iwata-1} are good examples.
In particular, Burgers equations, KdV equations, and mKdV equations are associated with those transforms (see \cite{88marchenko} for nonlinear equations in soliton theory). 
In conclusion the unbounded generalization of the logarithmic representation is presented by means of the doubly-implemented resolvent representation.

\section{Mathematical settings}
Let $X$ be a Banach space, and $t$ and $s$ be real numbers included in a finite interval $[0, T]$.
Let $U(t,s)$ be the evolution operator in $X$. 
Two parameter semigroup $U(t,s)$, which is continuous with respect to both parameters $t$ and $s$, is assumed to be a closed operator in $X$, and not necessarily a bounded operator on $X$.
That is, the boundedness condition 
\begin{equation} \label{bound} \begin{array}{ll}
\| U(t,s) \| \le M e^{\omega (t-s)},
\end{array} \end{equation}
is not assumed.
Furthermore the resolvent set of $U(t,s)$ is assumed to be non-empty.
Following the standard theory of abstract evolution equations~\cite{70kato,73kato,79tanabe} the semigroup property:
\[ \begin{array}{ll}
 U(t,r) U(r,s) = U (t,s)  
\end{array} \]
is assumed to be satisfied for arbitrary $s < r < t$ included in a finite interval $[0, T]$.
For any $t, r ,s \in [0, T]$ satisfying $s \le r \le t$, the domain space of $U(t,s)$ is assumed to be
\[ 
D(U(t,s)) = \{ x \in D(U(r,s)); ~  U(r,s) x  \in  D(U(t,r))  \}.
\]
Compared to the standard treatment, this condition is additionally assumed.
All the above conditions are weaker than those for the preceding work dealing with the logarithmic representation of infinitesimal generator of bounded $U(t,s)$, as the domain space is fixed to be $X$ in case of bounded $U(t,s)$~\cite{17iwata-1}.
Note that $U(t,s)$ can be either linear or nonlinear semigroup.

\section{Main result}
\subsection{Logarithm of unbounded two-parameter semigroup}
Let $X$ and $U(t,s)$ be a Banach space and a two-parameter semigroup respectively, as defined in the previous section.
In particular, $U(t,s)$ with its non-empty resolvent set is assumed to be only a closed operator but not a bounded operator.
For the definition of the logarithm of $U(t,s)$, the resolvent operator of $U(t,s)$
\begin{equation} \label{resolve} 
I_{\eta}(t,s) := ( I - \eta^{-1} U(t,s))^{-1}
\end{equation}
is utilized.
Here $\eta \in {\mathbf C}$ is assumed to be included in the resolvent set of $U(t,s)$, so that $I_{\eta}(t,s)$ is necessarily bounded on $X$.
It also follows according to the validity of
\begin{equation} \label{resolapp} \begin{array}{ll}
 I_{\eta}(t,s)  - I 
 =  ( I -I_{\eta}(t,s) ^{-1} ) I_{\eta}(t,s)  
 =  \eta^{-1} U(t,s)  I_{\eta}(t,s) 
\end{array} \end{equation}
that the product of operator $U(t,s)  I_{\eta}(t,s)$ is bounded on $X$. 
The operator $U(t,s)  I_{\eta}(t,s)$ is regarded as the resolvent approximation of $U(t,s)$, so that any difference between a certain resolvent operator and an identity operator can be regarded as the resolvent approximation (df. Yosida approximation).    
In the following the resolvent operator $I_{\eta}(t,s) $ is denoted simply by $I_{\eta}$, if there is no ambiguity.
Let ${\rm Log}$ be a principal branch of logarithm being defined by
\[ {\rm Log} z = \log |z| + i {\rm arg} Z, \]
where $Z$ is a complex number chosen to satisfy $|Z|=|z|$, $-\pi < {\rm arg} Z \le \pi$.
Formal calculation of logarithm leads to the definition of ${\rm Log} U(t,s)$
\begin{equation} \begin{array}{ll} \label{logddef}
{\rm Log} U(t,s) : =   {\rm Log} [ U(t,s) I_{\eta} ]  -  {\rm Log} [I_{\eta}] \vspace{2.5mm} \\
= \int_{\Gamma_1} {\rm Log} \lambda (\lambda I - U(t,s) I_{\eta})^{-1} d \lambda  
 - \int_{\Gamma_2} {\rm Log} \lambda (\lambda I - I_{\eta})^{-1} d \lambda, 
\end{array} \end{equation}
where the relation $U(t,s) = U(t,s) I_{\eta}  I_{\eta}^{-1} $ is adopted for the logarithmic rule.
The right hand side is always regarded as the Riesz-Dunford integral if it is possible to draw integral paths $\Gamma_1$ and $\Gamma_2$ excluding the origin and including the spectral sets of $ U(t,s) I_{\eta} $ and $I_{\eta}$ respectively.
It is worth noting here that $I_{\eta}^{-1}=I - \eta^{-1} U(t,s)$ is bounded if $\Gamma_2$ excludes the origin.
This fact leads to the contradicting situation that  Eq.~(\ref{logddef}) cannot be valid to define the logarithmic operator of unbounded $U(t,s)$. 
The essentially similar statement is true for $\Gamma_1$, as $U(t,s) I_\eta = \eta(I_{\eta}-I)$ is valid.
This contradicting situation is essentially settled by introducing the differential operator.
In order to obtain the logarithmic representation of infinitesimal generators of unbounded $U(t,s)$, it is useful to deal with the $t$-differential of logarithm.
Using Eq.~(\ref{resolapp}),
\begin{equation} \label{proposedrep} \begin{array}{ll} 
\partial_t  {\rm Log} U(t,s)  =  \partial_t  {\rm Log} [ \eta ( I_{\eta}  - I ) ]  - \partial_t   {\rm Log} [ I_{\eta}]  \vspace{2.5mm} \\
\qquad =  (I + \nu  \eta^{-1} ( I_{\eta}  - I )^{-1} ) \partial_t  {\rm Log} [ \eta  I_{\eta}  + (\nu -\eta) I ) ] 
- (I + \nu  I_{\eta}^{-1}  ) \partial_t   {\rm Log} [ I_{\eta}  + \nu  I ]
\end{array} \end{equation} 
follows, where $\nu$ is a certain complex number and $t$-differential is defined using a kind of weak topology (i.e. the locally-strong topology defined in \cite{20iwata-3}).
Here two different resolvent operators with parameters $\eta$ and $\nu$ are doubly implemented.
It is notable here that there is no need to take the limit such as $\eta \to \infty$ and/or  $\nu \to \infty$. 
It realizes the pure Riesz-Dunford treatment of logarithm of unbounded operators; the right hand side of Eq.~(\ref{proposedrep}) consists of the bounded type of logarithmic representation.
Since it is generally difficult to define $\partial_t {\rm Log} [ \eta ( I_{\eta}  - I ) ] $ and $\partial_t  {\rm Log}  [I_{\eta}]$ due to the absence of possible existence of integral paths, both operators are defined by introducing a translation $+\nu I$ on the complex plane. 
\begin{equation} \label{detail} \begin{array}{ll}
\partial_t {\rm Log} [ \eta ( I_{\eta}  - I ) ] 
:=  (I + \nu  \eta^{-1} ( I_{\eta}  - I )^{-1} ) \partial_t {\rm Log} [ \eta  I_{\eta}  + (\nu - \eta) I ) ]       \vspace{1.5mm} \\ 
\quad = (I + \nu  \eta^{-1} ( I_{\eta}  - I )^{-1} ) \partial_t \int_{\Gamma_3} {\rm Log} \lambda (\lambda I -  \eta  I_{\eta}  - (\nu - \eta) I )  )^{-1} d \lambda  ,     \vspace{2.5mm} \\
\partial_t  {\rm Log}  [I_{\eta}] 
:= (I + \nu  I_{\eta}^{-1}  )  \partial_t {\rm Log} [I_{\eta}  + \nu  I ]  \vspace{1.5mm} \\ 
\quad = (I + \nu  I_{\eta}^{-1}  )  \partial_t  \int_{\Gamma_4} {\rm Log} \lambda (\lambda I -  I_{\eta}  - \nu  I)  )^{-1} d \lambda,
\end{array} \end{equation}
where the right hand sides can be defined similar to the preceding work~\cite{17iwata-1}; i.e., the translation of spectral set in the complex plain.
Indeed, according to the boundedness of $ U(t,s) I_{\eta} = \eta ( I_{\eta}  - I )$ and $I_{\eta}$ for a certain $\eta$, there always exist certain $\nu$ for integral paths $\Gamma_3$  (parameterized by $\lambda$) and $\Gamma_4$ (parameterized by $\lambda$) excluding the origin and including the spectral sets of $ \eta  I_{\eta}  + (\nu -\eta) I $ and $I_{\eta} + \nu I$ respectively.
Regardless of the choice of $\eta$, the above two logarithms of operators in the right hand side are always well-defined by choosing appropriate $\nu$ with sufficiently-large $|\nu|$.
Since for $(I_{\eta}- I)^{-1}$ and $I_{\eta}^{-1}$ in Eqs. (\ref{proposedrep}) and (\ref{detail}), the resulting representation (\ref{logthm6}) in Lemma~\ref{thm1} is an intermediate representation, which is further modified to be fully bounded representation shown in Theorem~\ref{lem1} and Corollary~\ref{lem2}. 
In this way the logarithmic representation of infinitesimal generators is generalized step by step.
Since the representation using Riesz-Dunford integral is associated with the operators included in the $B(X)$-module \cite{18iwata, 20iwata-2}, here is an advantage of the proposed method only using the Riesz-Dunford integral representation.

\begin{lemma} \label{thm1}
Let $t$ and $s$ satisfy $0 \le s \le t \le T$, and $Y$ be dense subspace of $X$.
Let $U(t,s)$ be a continuous semigroup satisfying
\[ \begin{array}{ll}
 U(t,r) U(r,s) = U (t,s)  
\end{array} \]
for any $s < r < t$ included in a finite interval $[0, T]$
with its domain space
\[ 
D(U(t,s)) = \{ x \in D(U(r,s)); ~  U(r,s) x  \in  D(U(t,r))  \}.
\]
where  $D(U(t,s))$ is assumed to be a dense subspace of $X$.
Let $U(t,s)$ and its infinitesimal generator $A(t)$ be assumed to commute.
Furthermore $U(t,s)$ is assumed to be invertible (i.e., $U(s,t)= U(t,s)^{-1}$ exists), and the inverse of $\eta(I_{\eta}(t,s) - I)$ are assumed to exist as bounded on $X$.
For a given two-parameter closed operator $U(t,s): D(U) \to X$, its infinitesimal generator $A(t):Y \to X$ is represented by
\begin{equation} \label{logthm6} \begin{array}{ll}
A(t)  =  (I + \nu  \eta^{-1} ( I_{\eta}(t,s)  - I )^{-1} ) \partial_t    
 {\rm Log} [ I_{\eta}(t,s)  + \frac{\nu - \eta}{\eta} I  ]   \vspace{1.5mm} \\
- (I + \nu  I_{\eta}(t,s)^{-1}  )  \partial_t {\rm Log} [ I_{\eta} (t,s)  + \nu  I ],  
\end{array} \end{equation}
where $\nu$ and $\eta$ are certain complex numbers.
\end{lemma}

\begin{proof}
The logarithmic relation with $U(t,s) = U(t,s) I_{\eta}  I_{\eta}^{-1} =  \eta ( I_{\eta}  - I )  I_{\eta}^{-1} $ lead to
\[ \begin{array}{ll}
 A(t) =  \partial_t {\rm Log}  [ \eta ( I_{\eta}  - I )  I_{\eta}^{-1} ]  
       =  \partial_t {\rm Log} [ \eta ( I_{\eta}  - I ) ]  
- \partial_t  {\rm Log}  [I_{\eta}],  \vspace{1.5mm} \\
\end{array} \]
where all the logarithms in this equality are always well-defined by choosing $\eta$ property, although ${\rm Log}U(t,s)$ is not necessarily well-defined.
The well-definedness follows from the boundedness of $\eta(I_{\eta}-I)$ and $I_{\eta}$.
Using the logarithmic representation for bounded evolution operator~\cite{17iwata-1}, it is further calculated as
\[ \begin{array}{ll}
 A(t) =  (I + \nu  \eta^{-1} ( I_{\eta}  - I )^{-1} ) \partial_t {\rm Log} [ \eta  I_{\eta}  + (\nu - \eta) I ) ]   
- (I + \nu  I_{\eta}^{-1}  )  \partial_t {\rm Log} [ I_{\eta}  + \nu  I ], \vspace{1.5mm} \\
\quad =  (I + \nu  \eta^{-1} ( I_{\eta}  - I )^{-1} ) \partial_t {\rm Log} [ \eta ( I_{\eta}  + \frac{\nu - \eta}{\eta} I ) ]   
- (I + \nu  I_{\eta}^{-1}  )  \partial_t {\rm Log} [ I_{\eta}  + \nu  I ] \vspace{1.5mm} \\
\quad =  (I + \nu  \eta^{-1} ( I_{\eta}  - I )^{-1} ) \partial_t \left\{ {\rm Log}  [\eta I ]   
+ {\rm Log} [ I_{\eta}  + \frac{\nu - \eta}{\eta} I  ]    \right\}  
- (I + \nu  I_{\eta}^{-1}  )  \partial_t {\rm Log} [ I_{\eta}  + \nu  I ]  \vspace{1.5mm} \\
\quad =  (I + \nu  \eta^{-1} ( I_{\eta}  - I )^{-1} ) \partial_t    
 {\rm Log} [ I_{\eta}  + \frac{\nu - \eta}{\eta} I  ]
- (I + \nu  I_{\eta}^{-1}  )  \partial_t {\rm Log} [ I_{\eta}  + \nu  I ]  ,
\end{array} \]
where $I_{\eta}=I_{\eta}(t,s)$ is a function of $t$ and $s$.
\end{proof}

Here $\nu$ is taken to be sufficiently large to be included in the resolvent set of $I_{\eta}$ and $\eta (I_{\eta} - I)$.  
The operator $ (I + \nu  \eta^{-1} ( I_{\eta}  - I )^{-1} )$ is bounded, since $\eta( I_{\eta}  - I )$ is assumed to be bounded.
The operator $I + \nu I_{\eta}^{-1} = I_{\eta}^{-1} (I_{\eta} + \nu I)$ is bounded, since $\eta$ is included in the resolvent set of $U(t,s)$.
The boundedness for $\eta (I_{\eta} (t,s) -I)$ leads to the boundedness of $(I_{\eta}-I)^{-1}$; indeed, 
\[ \begin{array}{ll}
(I_{\eta} - I)^{-1} = \eta (I - \eta^{-1} U(t,s)) U(t,s)^{-1} \vspace{1.5mm} \\
= \eta  U(t,s)^{-1} -  U(t,s) U(t,s)^{-1}  \vspace{1.5mm} \\
= \eta  U(t,s)^{-1} - I |_{D(U(t,s)^{-1})}  \vspace{1.5mm} \\
= \eta  U(t,s)^{-1} - I |_{R(U(t,s))}.  
\end{array} \]
Note that the boundedness for $\eta (I_{\eta} (t,s) -I)$ is not removable, if the boundedness of $ (I + \nu  \eta^{-1} ( I_{\eta}  - I )^{-1} )$ is not satisfied.
The right hand side of Eq.~(\ref{logthm6}) plays a role of pre-infinitesimal generator being defined in Ref.~\cite{17iwata-1}.
 \\

\begin{corollary} \label{cor1}
Let $t$ and $s$ satisfy $0 \le s \le t \le T$, and $Y$ be dense subspace of $X$.
Let $U(t,s)$ be a continuous semigroup satisfying
\[ \begin{array}{ll}
 U(t,r) U(r,s) = U (t,s)  
\end{array} \]
for any $s < r < t$ included in a finite interval $[0, T]$
with its domain space
\[ 
D(U(t,s)) = \{ x \in D(U(r,s)); ~  U(r,s) x  \in  D(U(t,r))  \}.
\]
where  $D(U(t,s))$ is assumed to be a dense subspace of $X$.
Let $U(t,s)$ and its infinitesimal generator $A(t)$ be assumed to commute.
Furthermore $U(t,s)$ is assumed to be invertible, and the inverse of $\eta(I_{\eta}(t,s) - I)$ are assumed to exist as bounded on $X$.
In particular, $\eta/(1-\eta)$ is assumed to be included in the resolvent set of $I_{\eta}$ and $\eta (I_{\eta} - I)$.
For a given two-parameter closed operator $U(t,s): D(U) \to X$, its infinitesimal generator $A(t):Y \to X$ is represented by
\begin{equation} \label{logthm2} \begin{array}{ll}
A(t)  =  ( I_{\eta}(t,s)^2  -  I_{\eta}(t,s) )^{-1}   ( I_{\eta}(t,s)  + \nu I  )  
  \partial_t {\rm Log} [ I_{\eta}(t,s)  + \nu  I ],
\end{array} \end{equation}
where $\nu$ and $\eta$ are certain complex numbers satisfying 
 \begin{equation} \label{equival}    \nu   = \frac {\eta}{1  - \eta}.  \end{equation}

\end{corollary}

\begin{proof}
According to Lemma~\ref{thm1},
\[ \begin{array}{ll}
 A(t) =  (I + \nu  \eta^{-1} ( I_{\eta}  - I )^{-1} ) \partial_t    
 {\rm Log} [ I_{\eta}  + \frac{\nu - \eta}{\eta} I  ]
- (I + \nu  I_{\eta}^{-1}  )  \partial_t {\rm Log} [ I_{\eta}  + \nu  I ] 
\end{array} \]
is valid under the assumption, where $I_{\eta}=I_{\eta}(t,s)$ is a function of $t$ and $s$.
According to the assumption, $\nu$ is possible to be taken as $\nu   = \eta/(1  - \eta)$. 
The logarithms in the right hand side is always well-defined for such $\nu$. 
It follows that
\[ \begin{array}{ll}
 A(t) =  [ (I + \nu  \eta^{-1} ( I_{\eta}  - I )^{-1} ) - (I + \nu  I_{\eta}^{-1}  ) ] 
  \partial_t {\rm Log} [ I_{\eta}  + \nu  I ] \vspace{1.5mm} \\
   =  [  I_{\eta}  + \eta (  1 -  \eta)^{-1} I  ] \eta^{-1} (  1 -  \eta) \nu  I_{\eta}^{-1}  ( I_{\eta}  - I )^{-1} 
  \partial_t {\rm Log} [ I_{\eta}  + \nu  I ],
\end{array} \]
so that the logarithmic representation is obtained as
\[ \begin{array}{ll}
 A(t)  
 =     ( I_{\eta}^2  -  I_{\eta} )^{-1}   ( I_{\eta}  + \nu I  )  
  \partial_t {\rm Log} [ I_{\eta}  + \nu  I ],
\end{array} \]
where Eq.~(\ref{equival}) is utilized.  

For the existence of 
\begin{equation}  \begin{array}{ll}
( I_{\eta}^2  -  I_{\eta} )^{-1}   ( I_{\eta}  + \nu I  )  
= I_{\eta}^{-2}  ( I  -  I_{\eta}^{-1} )^{-1}   ( I_{\eta}  + \nu I  ) 
\end{array} \end{equation}
as a bounded operator on $X$, it is necessary for 1 to be included in the resolvent set of $I_{\eta}^{-1}$.
The condition $1 \in \rho( I_{\eta}^{-1})$ is satisfied by the assumption: the inverse of $\eta(I_{\eta} - I)$ are assumed to exist as bounded on $X$.
\end{proof}

The operator $( I_{\eta}(t,s)  + \nu  I )$, whose inverse operator is not necessarily bounded, is the inverse of the resolvent operator of $I_{\eta}(t,s) $, here one can see a reason why the resolvent representation is doubly-imposed in the resulting logarithmic representation.
Note that the right hand side of Eq.~(\ref{logthm2}) plays a role of pre-infinitesimal generator being defined in Ref.~\cite{17iwata-1}.
A pre-infinitesimal generator is an operator possible to be an infinitesimal generator if an ideal domain is given. \\

\subsection{Logarithmic representation using the alternative infinitesimal generator}
The invertible assumption in Lemma~\ref{thm1} is removed by introducing the alternative infinitesimal generator~\cite{17iwata-3}.
Indeed, it is readily seen that $0$ and $1$ are included in the resolvent set of $I_{\eta}^{-1}$ by replacing $I_{\eta}$ with $e^{a(t,s)} - \nu  I$ for a certain sufficiently large $|\nu|$.
The logarithmic representation for infinitesimal generator of closed operator shown in Lemma~\ref{thm1} is generalized in the following Corollary.

\begin{theorem} \label{lem1}
Let $\partial_t a_i(t,s)$ be alternative infinitesimal generators ($i=1,2$) being defined by 
\begin{equation} \begin{array}{ll}
a_1(t,s) := {\rm Log} [ \eta (I_{\eta} (t,s) - I) ], \vspace{1.5mm}\\
a_2(t,s) := {\rm Log} [ I_{\eta}(t,s)  + \nu  I ],
\end{array} \end{equation} 
respectively.
Let $t$ and $s$ satisfy $0 \le s \le t \le T$, and $Y$ be dense subspace of $X$.
Let $U(t,s)$ be a continuous semigroup satisfying
\[ \begin{array}{ll}
 U(t,r) U(r,s) = U (t,s)  
\end{array} \]
for any $s < r < t$ included in a finite interval $[0, T]$
with its domain space
\[ 
D(U(t,s)) = \{ x \in D(U(r,s)); ~  U(r,s) x  \in  D(U(t,r))  \}.
\]
where  $D(U(t,s))$ is assumed to be a dense subspace of $X$.
Let $U(t,s)$ and its infinitesimal generator $A(t)$ be assumed to commute.
For a given two-parameter closed operator $U(t,s): D(U) \to X$, its infinitesimal generator $A(t)$ is represented by
\begin{equation} \label{logthm} \begin{array}{ll}
A(t)  = (I + \nu  e^{-a_1(t,s)}) \partial_t a_1(t,s)
-   (I + \nu  e^{-a_2(t,s)})  \partial_t  a_2(t,s), 
\end{array} \end{equation}
where $\nu$ and $\eta$ are certain complex number.
\end{theorem}

\begin{proof}
It follows from Eq.~(\ref{detail}) and Theorem~\ref{thm1}.
\end{proof}

The relation $I_{\eta}(t,s) = e^{a_1(t,s)} - \nu  I$ leads to the boundedness 
\[
 \| e^{a_1(t,s)} \| \le \| I_{\eta} \| + \nu. \] 
Simultaneously its inverse $e^{-a_1(t,s)} =  (I_{\eta}  + \nu I)^{-1}$ is bounded, since $\nu$ is assumed to be included in the resolvent set of $I_{\eta}$. 
In conclusion, operators $a_1(t,s)$ and $a_2(t,s)$, which are well-defined by properly chosen $\nu$ in Eq.~(\ref{detail}), are bounded.
For the representation (\ref{logthm}), it is not necessary to assume the invertible properties of $U(t,s)$ and $\eta(I_{\eta}(t,s) - I)$.
This is a different point compared to the original Lemma~\ref{thm1}.

\begin{corollary} \label{lem2}
Let $\partial_t a(t,s)$ be an alternative infinitesimal generator being defined by 
\begin{equation}
a(t,s) := {\rm Log} [ I_{\eta}(t,s)  + \nu  I ].
\end{equation} 
Let $t$ and $s$ satisfy $0 \le s \le t \le T$, and $Y$ be dense subspace of $X$.
Let $U(t,s)$ be a continuous semigroup satisfying
\[ \begin{array}{ll}
 U(t,r) U(r,s) = U (t,s)  
\end{array} \]
for any $s < r < t$ included in a finite interval $[0, T]$
with its domain space
\[ 
D(U(t,s)) = \{ x \in D(U(r,s)); ~  U(r,s) x  \in  D(U(t,r))  \}.
\]
where  $D(U(t,s))$ is assumed to be a dense subspace of $X$.
Let $U(t,s)$ and its infinitesimal generator $A(t)$ be assumed to commute.
In particular, $\eta/(1-\eta)$ is assumed to be included in the resolvent set of $I_{\eta}$ and $\eta (I_{\eta} - I)$.
For a given two-parameter closed operator $U(t,s): D(U) \to X$, its infinitesimal generator $A(t)$ is represented by
\begin{equation} \label{loglem} \begin{array}{ll} 
A(t)    =      \left( e^{a(t,s)} - (2 \nu + 1) I  + ( \nu^2   + \nu)   e^{-a(t,s)}   \right)^{-1} \partial_t a(t,s),
\end{array} \end{equation}
where $\nu$ and $\eta$ are certain complex numbers satisfying $\nu   = \eta/(1  - \eta)$.
\end{corollary}

\begin{proof}
Following the previous research~\cite{17iwata-3}, it is practical to define the alternative infinitesimal generator $\partial_t a(t,s)$ with the relation $a(t,s) = {\rm Log} [ I_{\eta}(t,s)  + \nu  I ]$.
Since $a(t,s)$ is a bounded operator on $X$, there is no restriction to define both $e^{a(t,s)}$ and $e^{-a(t,s)}$ simultaneously. 
The logarithmic representation of the infinitesimal generators shown in Lemma~\ref{thm1} is replaced using the alternative infinitesimal generator.
\[ \begin{array}{ll}
 A(t) 
   =      \left( e^{a(t,s)} - (2 \nu + 1) I  + ( \nu^2   + \nu)   e^{-a(t,s)}   \right)^{-1} 
  \partial_t a(t,s)
\end{array} \]
follows, where $\partial_t a(t,s)$ is the infinitesimal generator of
\begin{equation} \begin{array}{ll}
e^{a(t,s)} =  I_{\eta}(t,s)  + \nu  I = (I-\eta^{-1}U(t,s))^{-1} \{ (\nu+1) I- \nu \eta^{-1}U(t,s)) \},
\end{array} \end{equation}
so that the infinitesimal generator of the inverse of doubly-implemented resolvent operator $ I_{\eta}(t,s)  + \nu  I $.
 \end{proof}
 
Operators $a(t,s)$, which are well-defined by properly chosen $\nu$ in Eq.~(\ref{detail}), are bounded.
For the representation (\ref{loglem}), it is not necessary to assume the invertible properties of $U(t,s)$ and $\eta(I_{\eta}(t,s) - I)$.
This is a different point compared to the original Corollary~\ref{cor1}. \\

\subsection{Algebraic property}
Under the well-definedness of $a(t,s) := {\rm Log} [ I_{\eta}(t,s)  + \nu  I ]$, there are three conditions to be satisfied for $A(t)$ to be an element of $B(X)$-module~\cite{18iwata, 20iwata-2}:
\begin{enumerate}
\item boundedness of $I + \nu e^{-a_1(t,s)}$ and  $I + \nu e^{-a_2(t,s)}$ in Eq.~(\ref{logthm});  \\
\item continuity of $\nu e^{-a_1(t,s)}$ and  $\nu e^{-a_2(t,s)}$ with respect to $t$ and $s$; \\
\item commutation between $e^{-a_i(t,s)}$ and $\partial_t a_i(t,s)$ respectively for $i=1,2$.  
\end{enumerate}
It is remarkable that parts $\nu e^{-a_1(t,s)}$ and $\nu e^{-a_2(t,s)}$ are always bounded on $X$, as $e^{a_i(t,s)}$ and $e^{-a_i(t,s)}$ has been clarified to be bounded on $X$.
Furthermore this part is continuous with respect to $t$ and $s$, as $U(t,s)$ is assumed to be continuous with respect to $t$ and $s$. 
Indeed it is sufficient to see that \vspace{1.5mm}  \\
\qquad continuity of $U(t,s)$ 
\quad $\Rightarrow$  \quad  continuity of $I_{\eta}(t,s)$  
\quad $\Rightarrow$  \quad  continuity of $a(t,s)$ \vspace{1.5mm} \\
is true for both $t$ and $s$, where $\eta \in {\mathbf C}$ is taken from the resolvent set of $U(t,s)$.
For the commutation, since the commutation between $U(t,s)$ and $I_{\lambda}(t,s)$ is always true, \vspace{1.5mm} \\
\qquad commutation: $U(t,s)$ and $A(t)=\partial_t U(t,s)$ \vspace{1.5mm} \\
 \qquad \qquad $\Rightarrow$  \quad   commutation: $I_{\lambda}(t,s)$ and $\partial_t I_{\lambda}(t,s)$  \vspace{1.5mm} \\
 \qquad \qquad $\Rightarrow$  \quad   commutation: $I_{\lambda}(t,s)$ and $\partial_t a(t,s)$, \vspace{1.5mm} \\
is satisfied, where the commutation between $U(t,s)$ and $A(t)$ is assumed (cf. assumption of Lemma~\ref{thm1} and Theorem~\ref{lem1}). 
Consequently, the infinitesimal generators of unbounded evolution operators shown in Eq.~(\ref{logthm}) are clarified to be elements of $B(X)$-module; i.e. a module over the Banach algebra $B(X)$.
This is a generalization of original logarithmic representation in the sense of extending the applicable evolution operators. \\

\subsection{Relativistic formulation}
In terms of applying to nonlinear transform such as the Cole-Hopf transform\cite{51cole,50hopf}, it is useful to introduce the relativistic form of the obtained logarithmic representation \cite{19iwata-1}.
The relativistic formulation is represented using the tensor notation.

\begin{corollary} 
Let $n$ be a positive integer, and $i$ be the evolution direction ($1 \le i \le n$).
Let $ x^i$ and $\xi^i$ satisfy $-L_ \le x^i, \xi^i \le L$, and $Y$ be dense subspace of $X$.
Let $U(x^i,\xi^i)$ be a continuous semigroup satisfying
\[ \begin{array}{ll}
 U(x^i,\eta^i) U(\eta^i,\xi^i) = U (x^i,\xi^i) 
\end{array} \]
for any $\xi^i < \eta^i < x^i$ included in a finite interval $[-L, L]$
with its domain space
\[ 
D(x^i,\xi^i) = \{ v \in D(U(\eta^i, \xi^i)); ~  U(\eta^i,\xi^i) v \in  D(U(x^i,\eta^i))  \}.
\]
where  $D(U(x^i,\xi^i))$ is assumed to be a dense subspace of $X$.
Let $U(x^i,\xi^i)$ and its infinitesimal generator $A(t)$ be assumed to commute.

For a given two-parameter closed operator $U(x^i,\xi^i): D(U) \to X$, its infinitesimal generator $K(x^i)$ is represented by
\begin{equation}  \begin{array}{ll}
K(x^i)  = (I + \nu  e^{-a_1(x^i,\xi^i)}) \partial_t a_1(x^i,\xi^i)
-   (I + \nu  e^{-a_2(x^i,\xi^i)})  \partial_t  a_2(x^i,\xi^i),
\end{array} \end{equation}
where $\nu$ and $\eta$ are certain complex numbers satisfying $\nu   = \frac {\eta}{1  - \eta}$, and $a_i(x^i,\xi^i)$ is an alternative infinitesimal generator being defined by
\begin{equation} \begin{array}{ll}
a_1(t,s) := {\rm Log} [ \eta (I_{\eta}(t,s)  - I) ], \vspace{1.5mm}\\
a_2(t,s) := {\rm Log} [ I_{\eta}(t,s)  + \nu  I ].
\end{array} \end{equation}  
In particular, if $\eta/(1-\eta)$ is assumed to be included in the resolvent set of $I_{\eta}$ and $\eta (I_{\eta} - I)$.
\begin{equation} \begin{array}{ll}
K(x^i)    =      \left( e^{a_1(x^i,\xi^i)} - (2 \nu + 1) I  + ( \nu^2   + \nu)   e^{-a_1(x^i,\xi^i)}   \right)^{-1} \partial_{x^i} a_1(x^i,\xi^i)
\end{array} \end{equation}
is true.
\end{corollary}

\begin{proof}
The statement follows from Theorem~\ref{lem1} and Corollary~\ref{lem2}.
\end{proof}

In this formulation the evolution direction $x^i$ ($i$-th direction in the tensor form) can be either $t$ or others.

\section{Logarithm of strip-type operator} \label{sec-iden}
Let a sectorial operator ${\mathcal U}$ be an injective operators in a Banach space $X$. 
The strip-type operator \cite{06hasse} is associated with the logarithm of sectorial operators; indeed, the logarithm of sectorial operator ${\mathcal U}$ is a strong strip-type operator ${\mathcal A} =  {\rm Log}  ~ {\mathcal U}$ with a finite height.
The strip-type operators do not satisfy the sectorial property if they are unbounded, while the logarithm of strip-type operators can be always well-defined by the present method. 
That is, the logarithm of a certain class of stripe-type operator is the logarithm of logarithm:
\[ {\rm Log} {\mathcal A} : =  {\rm Log} ( {\rm Log} ~ {\mathcal U}),  \]
which can be defined by the present method, where ${\mathcal U}$ is a sectorial operator, but ${\mathcal A}$ is not necessarily a sectorial operator. 
This equation is re-written as
\[ e^{ {\rm Log} {\mathcal A} } =  {\rm Log} ~ {\mathcal U},  \]
if it is possible to define the exponential of both sides.
In some cases an operator $ {\rm Log} ~ {\mathcal U} $ corresponding to the infinitesimal generator of $ {\mathcal U} $ is formally identified with the evolution operator $e^{ {\rm Log} {\mathcal A} }$.

This situation is applicable to the Cole-Hopf transform
\[  {\mathcal A} = -2  \mu^{-1/2} \partial_x \log {\mathcal U} 
  = -2  \mu^{-1/2}     {\mathcal U}^{-1}  \partial_x {\mathcal U} ,
  \]
which is equivalent to a differential equation
\[ 
   { \partial_x {\mathcal U} }  = - ( \mu^{1/2} /2) {\mathcal A} ~  {\mathcal U},
  \]
where it is clear from this relation that $ {\mathcal A} $ and  $ {\mathcal U} $ play roles of infinitesimal generator and the evolution operator, respectively. 
On the other hand, according to the original meaning of the Cole-Hopf transform, $ {\mathcal A} $ and $ {\mathcal U} $ are the solutions of Burgers equation and heat equation, respectively.
Therefore it is necessary to take an infinitesimal generator $ {\mathcal A} $ as an evolution operator (a solution of differential equation) in this context, and then the possibility of defining another evolution operator $ e^{ {\rm Log} {\mathcal A} } $ settle the problem.
It leads to the abstract and closed (or self-consistent) framework of nonlinear transform, which appears in soliton theory.

\section{Summary}
In the application to soliton theory, it is practical to consider that the evolution operator $U(t,s)$ to be an infinitesimal generator of another equation.
For this purpose, it is often necessary to have a certain kind of identification between infinitesimal generators and evolution operators (Sec.~\ref{sec-iden}). 
This situation is often found in the theory of integrable systems such as soliton theory (for example, see \cite{73scott}).
The general framework for identifying the evolution operators (i.e. solutions) with the infinitesimal generators of another equations (an operator included in an equation) is established in this paper.
By focusing on the  logarithmic relation, the representation is obtained only using the Riesz-Dunford integral of resolvent operators.
From an algebraic point of view, the logarithmic representation shown in the presented paper is included in the $B(X)$-module.

In conclusion, the standard theory of abstract evolution equations itself is generalized in the sense of weakening the assumptions for evolution operators, and the theory for logarithm of operators is improved in the sense of removing the sectorial assumption.
The present framework of treating the logarithmic representation of generally-unbounded two parameter semigroup (evolution operator) will open up a way to analyze
\begin{itemize}
\item abstract and general treatments to linear and nonlinear transforms; 
\item explosions of solutions at a finite time;
\item some stochastic differential equations without assuming the bounded time evolution;
\end{itemize}
where note that two parameter semigroup $U(t,s)$ can be either linear or nonlinear semigroups.
It is sometimes possible to remove the bounded time interval $0 \le t,s \le T$ or $-L \le x^i , \xi^i \le L$ assumptions in main results, because it is originally assumed for ensuring the boundedness of operator \cite{17iwata-1}; e.g. the boundedness of $I_{\eta}(t,s)$ in the present paper.

\end{document}